\pdfoutput=1
\documentclass[11pt]{article}

\usepackage[a4paper,margin=1in]{geometry}
\usepackage[T1]{fontenc}

\usepackage{amsmath,amssymb,amsfonts,amsthm}

\usepackage{graphicx}
\usepackage{algorithm,algorithmic}
\usepackage{float}
\usepackage{textcomp}
\usepackage{subcaption}

\usepackage{comment}
\usepackage{cite}
\usepackage{hyperref}
\hypersetup{
  hidelinks,
  pdftitle={Weighted Laplacian Flow: A Deterministic Particle Flow with Exponential Convergence},
  pdfauthor={Weiye Gan, Tangjun Wang, Zuoqiang Shi}
}
\usepackage{cleveref}

\newtheorem{theorem}{Theorem}[section]
\newtheorem{lemma}[theorem]{Lemma}
\newtheorem{proposition}[theorem]{Proposition}

\theoremstyle{remark}
\newtheorem{remark}[theorem]{Remark}

\crefname{theorem}{Theorem}{Theorems}
\crefname{proposition}{Proposition}{Propositions}
\crefname{lemma}{Lemma}{Lemmas}
\crefname{remark}{Remark}{Remarks}
\crefname{figure}{Fig.}{Figs.}
\crefname{assumption}{Assumption}{Assumptions}
\crefname{corollary}{Corollary}{Corollaries}

\title{Weighted Laplacian Flow: A Deterministic Particle Flow with Exponential Convergence%
\thanks{This work was supported by National Natural Science Foundation of China (NSFC) 92370125.}}

\author{
Weiye Gan\thanks{Department of Mathematical Sciences, Tsinghua University,
Beijing 100084, China. \texttt{gwy22@mails.tsinghua.edu.cn}}
\and
Tangjun Wang\thanks{Department of Mathematics, The University of Hong Kong,
Pokfulam Road, Hong Kong SAR, China. \texttt{wangtj@hku.hk}}
\and
Zuoqiang Shi\thanks{Yau Mathematical Sciences Center, Tsinghua University,
Beijing 100084, China; and Yanqi Lake Beijing Institute of Mathematical
Sciences and Applications, Beijing 101408, China.
Corresponding author: \texttt{zqshi@tsinghua.edu.cn}}
}

\date{}

\begin{document}

\maketitle

\begin{abstract}
We introduce weighted Laplacian flow (WLF), a deterministic particle-flow framework for sampling from a target distribution known only up to normalization. The proposed method evolves the logarithmic density ratio through a transport equation and determines the particle velocity from a target-weighted Poisson problem, resulting in a nonlocal and kernel-free mechanism for redistributing mass. For bounded smooth domains, we establish global well-posedness of the flow for Lipschitz initial data and prove that the flow contracts the discrepancy between the evolving and target densities at an exact exponential rate in terms of the oscillation of the log-density ratio, as well as in the associated projective metric. The relative entropy also satisfies a precise dissipation relation involving the two directional Kullback--Leibler divergences.
In addition, we identify a variational interpretation of WLF. The method can be viewed as a gradient flow of the reverse Kullback--Leibler divergence under a target-anchored metric. Consequently, the method achieves an explicit relaxation scale without requiring log-concavity or spectral-gap assumptions on the target distribution. Numerical experiments on multimodal, heavy-tailed, and high-dimensional targets demonstrate the effectiveness of the proposed approach in capturing long-range mass transport.
\end{abstract}


\medskip

\section{Introduction}
\label{sec:introduction}

Sampling from a target probability density is a fundamental task in statistics,
machine learning, and scientific computing
\cite{robert2004monte,gelman2013bayesian,durmus2019high,song2019generative}.
In many applications, the target density $p$ is available only up to a
normalizing constant, and the objective is to generate particles whose empirical
distribution approximates $p$. A central question is therefore how to construct
an efficient evolution that transports an accessible initial distribution
$q_0$ toward the target while retaining a tractable analytical structure.

A classical approach is based on stochastic diffusion. The variational
formulation of Jordan, Kinderlehrer and Otto~\cite{jko1998} shows that the
Fokker--Planck equation associated with overdamped Langevin dynamics can be
interpreted as the Wasserstein gradient flow of the relative entropy
$\mathrm{KL}(\cdot\,\|\,p)$
\cite{ambrosio2005gradient,santambrogio2017euclidean}. At the density level,
\begin{equation}
\label{eq:langevin}
    \frac{\partial q_t}{\partial t}
    =
    \nabla\cdot\left(q_t\nabla\log\frac{q_t}{p}\right).
\end{equation}
The corresponding deterministic probability-flow velocity is
\[
    v_{\mathrm{Lan}}
    =
    \nabla\log\frac{p}{q_t},
\]
while the same Fokker--Planck equation can be realized by the Langevin diffusion
\[
    \mathrm{d}X_t
    =
    \nabla\log p(X_t)\,\mathrm{d}t
    +
    \sqrt{2}\,\mathrm{d}W_t .
\]
Its convergence has been studied extensively under assumptions such as
log-concavity or log-Sobolev inequalities
\cite{roberts1996exponential,durmus2017nonasymptotic,dalalyan2017theoretical,wang2020exponential,yang2025non}.
For multimodal or heavy-tailed targets, however, purely diffusive exploration
may suffer from slow barrier crossing and inefficient long-range mass transport.

Deterministic particle methods provide an alternative by transporting particles
through an interacting velocity field. A representative example is Stein
Variational Gradient Descent (SVGD)~\cite{liu2016svgd}, whose mean-field limit is
the Stein variational gradient flow
\cite{liu2017stein,lu2019scaling,duncan2023geometry}
\begin{equation}
\label{eq:SVGD1}
    \frac{\partial q_t}{\partial t}
    =
    \nabla\cdot\left(
    q_t\mathcal{T}_{k,q_t}\nabla\log\frac{q_t}{p}
    \right),
\end{equation}
where
\[
    \mathcal{T}_{k,q_t}f(x)
    =
    \int k(x,y)f(y)q_t(y)\,\mathrm{d}y .
\]
SVGD has been widely used in machine learning and applied mathematics
\cite{liu2017policy,chewi2020svgd,korba2020non,wang2019stein,wang2018stein}.
Its behavior depends on the choice of kernel, and the kernel integral operator
modifies the underlying Wasserstein force. The regularized Stein variational
gradient flow of He et al.~\cite{he2024regularized} further clarifies the
connection between Stein and Wasserstein dynamics.

The present work starts from a different modeling principle. Let $\rho$ be a
target density known only up to a multiplicative constant and set
\[
    \omega(t,x)=\log\frac{\rho(x)}{q_t(x)}.
\]
If particles follow a deterministic velocity field $v$, then the logarithmic
density ratio satisfies
\[
    \frac{\mathrm{d}}{\mathrm{d}t}\omega(t,X_t)
    =
    \frac{1}{\rho}\nabla\cdot(\rho v).
\]
Restricting to $v=\nabla\phi$ and prescribing the centered linear relaxation
\[
    \frac{\mathrm{d}}{\mathrm{d}t}\omega(t,X_t)
    =
    -(\omega(t,X_t)-\bar\omega(t))
\]
leads to the weighted Poisson equation
\[
    \mathcal{L}_\rho\phi
    :=
    \frac1\rho\nabla\cdot(\rho\nabla\phi)
    =
    -(\omega-\bar\omega).
\]
The scalar $\bar\omega(t)$ is determined by the Neumann compatibility
condition. Since $\mathcal{L}_\rho$ is unchanged when $\rho$ is multiplied by a
positive constant, the resulting dynamics are intrinsically
normalization-free.

Besides this direct relaxation interpretation, WLF possesses a
variational structure that distinguishes it from Wasserstein/Langevin dynamics.
For interpretation only, normalize $p=\rho/\int_\Omega\rho$ and write
$\eta=\log(p/q)$. If
\[
    \mathcal D_p v:=\frac1p\nabla\cdot(pv),
    \qquad
    P_p g:=g-\int_\Omega pg\,dx,
\]
then the Neumann compatibility condition gives
\[
    \bar\omega-\log\!\int_\Omega\rho
    =
    \int_\Omega p\eta\,dx
    =
    \mathrm{KL}(p\,\|\,q),
\]
and the WLF equation becomes
\[
    \mathcal D_p v=-P_p\eta.
\]
We show that this is precisely the formal gradient flow of the relative entropy
$\mathrm{KL}(p\,\|\,q)$ under a target-weighted metric that measures
the weighted compression $\mathcal D_pv$, rather than the kinetic velocity
itself. The weighted Poisson equation then has a second variational
interpretation: among all velocity fields producing the optimal compression
rate, its gradient solution has minimum weighted kinetic energy. Thus the
centering term is not merely a solvability correction; it is the orthogonal
projection of the $\mathrm{KL}(p\,\|\,q)$ force onto the zero-mean compression space.

The construction is related to, but different from, several established
geometries on probability measures. Dynamical transport distances with a
reference measure or nonlinear mobility modify the local transport action
\cite{dolbeault2009transport,carrillo2010nonlinear}, while Hessian transport
metrics build an inverse metric operator from the transported inverse Hessian
of an entropy functional \cite{li2019hessian}. WLF has a different
target-anchored operator structure: the current density $q$ determines the
transport kinematics, whereas the fixed target $p$ determines the compression
geometry. This distinction is reflected in the target-anchored operator viewpoint
developed in \Cref{sec:variational-geometry}.

The main contributions of this paper are as follows.
\begin{itemize}
    \item We introduce WLF as a kernel-free deterministic particle flow based on
    controlled relaxation of the logarithmic density ratio and formulate the
    coupled elliptic-transport system \eqref{eq:system}. 

    \item We establish global well-posedness of the WLF PDE system for
    Lipschitz initial log-density ratios on a bounded smooth domain with
    homogeneous Neumann boundary conditions. The solution remains spatially
    Lipschitz on every finite time interval. The proof
    combines a Schauder fixed-point construction at the H\"older level,
    maximum-principle estimates, logarithmic elliptic control and propagation of
    Lipschitz regularity.

    \item We obtain several exact convergence identities. The log-density-ratio
    oscillation contracts exactly as $e^{-t}$. In addition,
    the dissipation rate of $\mathrm{KL}(q_t\,\|\,p)$ is exactly the Jeffreys divergence, $-\mathrm{KL}(q_t\,\|\,p)-\mathrm{KL}(p\,\|\,q_t)$.
    Hence $\mathrm{KL}(q_t\,\|\,p)$ also decays exponentially. Importantly, this convergence result requires no logarithmic Sobolev inequality, Poincaré inequality, convexity, or curvature-type assumption on the target distribution.

    \item We identify a variational structure of WLF. We
    introduce a target-weighted divergence metric on the manifold of positive
    probability densities and show that WLF is the formal gradient flow of
    $\mathrm{KL}(p\,\|\,q)$ in this metric. At equilibrium, the metric
    reduces to the $L^2(p)$ metric on relative density perturbations, and the
    linearized WLF is simply $\partial_t u=-u$.
\end{itemize}

The remainder of the paper is organized as follows. \Cref{sec:settings}
derives the mathematical model of the weighted Laplacian flow. \Cref{sec:well-posedness} states the
global well-posedness result for Lipschitz initial data, and
\Cref{sec:convergence} establishes the exact contraction and entropy identities.
\Cref{sec:variational-geometry} then develops the variational and geometric
structure underlying these dynamics. The detailed existence, regularity, and
stability proofs are given in \Cref{sec:proof-well-posedness}. Numerical
experiments are reported in \Cref{sec:numerical-experiments}, and concluding
remarks are given in \Cref{sec:conclusion}.

\section{Formulation of Weighted Laplacian Flow}
\label{sec:settings}

Let $\Omega\subset\mathbb{R}^d$ be a bounded smooth domain and $\rho:\bar\Omega\to(0,\infty)$ be a smooth strictly positive unnormalized target density (equivalently, a
positive target weight) known only up to a multiplicative constant, and let
$q(t,x)$ be the probability density of particles
transported by
\[
    \frac{dX_t}{dt}=v(t,X_t).
\]
Then $q$ satisfies the continuity equation
\begin{equation}
\partial_t q+\nabla\cdot(qv)=0.
\label{eq:continuity}
\end{equation}
We introduce the logarithmic density ratio
\[
    \omega(t,x)=\log\frac{\rho(x)}{q(t,x)}.
\]

Differentiating $\omega$ along particle trajectories gives
\begin{equation}
\frac{d}{dt}\omega(t,X_t)
=
\frac1{\rho}\nabla\cdot(\rho v)(t,X_t).
\label{eq:omega_transport}
\end{equation}
The main idea is to construct the velocity field such that $\omega$ converges to a constant. One natural approach is to restrict the velocity to a gradient field $v=\nabla\phi$ and prescribe the
centered linear relaxation
\[
    \frac{d}{dt}\omega(t,X_t)=-(\omega(t,X_t)-\bar\omega(t)).
\]
This leads to the weighted Poisson equation
\begin{equation}
\frac1{\rho}\nabla\cdot(\rho\nabla\phi)
=
-(\omega-\bar\omega).
\label{eq:poisson}
\end{equation}
Here $\bar\omega(t)$ is treated as an additional scalar unknown and is
determined together with $\phi$ by the solvability condition for the Neumann
problem,
\begin{equation}
    \int_\Omega \rho(x)\bigl(\omega(t,x)-\bar\omega(t)\bigr)\,dx=0.
\label{eq:bar-solvability}
\end{equation}
Equivalently,
\[
    \bar\omega(t)
    =
    \frac{\int_\Omega \rho(x)\omega(t,x)\,dx}
         {\int_\Omega \rho(x)\,dx}.
\]
This identity is used only analytically: the formulation remains invariant under
multiplication of $\rho$ by an arbitrary positive constant. Thus neither the
model nor its numerical implementation requires the partition function.

Now we can write the complete mathematical model of the proposed weighted Laplacian flow as following:
\begin{equation}
\left\{
\begin{aligned}
&
\frac1{\rho}\nabla\cdot(\rho\nabla\phi)
=
-(\omega-\bar\omega),
\\
&
\frac{d}{dt}\omega(t,X(t,x))
=
-(\omega(t,X(t,x))-\bar\omega(t)),
\\
&
\frac{dX(t,x)}{dt}
=
\nabla\phi(t,X(t,x)),
\\
&
\frac{\partial\phi}{\partial n}=0,\qquad
\text{on } \partial\Omega
\\
&
\int_\Omega \rho(x)\bigl(\omega(t,x)-\bar\omega(t)\bigr)\,dx=0,\quad \int_\Omega\phi(t,x)\,dx=0
\end{aligned}
\right.
\label{eq:system}
\end{equation}

\begin{remark}
\label{remark1}
The homogeneous Neumann boundary condition ensures that the particle velocity
$v=\nabla\phi$ has no normal component at the boundary, so particles remain
confined in $\Omega$. Periodic boundary conditions can be treated similarly.
\end{remark}

\begin{remark}
If $\rho$ is replaced by $c\rho$ with $c>0$, then $\omega$ and
$\bar\omega$ are both shifted by $\log c$, whereas
$\omega-\bar\omega$, $\phi$, and the particle velocity are unchanged. Hence
the flow is intrinsically normalization-free.
\end{remark}

\section{Global Well-Posedness}
\label{sec:well-posedness}

Throughout this section, let $\Omega\subset\mathbb{R}^d$ be a bounded domain
with smooth boundary and let $\rho\in C^\infty(\bar\Omega)$ be strictly
positive.

\begin{theorem}[Global well-posedness for Lipschitz data]
\label{thm:global_wellposedness}
Let
\[
    \omega_0\in C^{0,1}(\bar\Omega).
\]
Then the WLF system \eqref{eq:system} admits a unique global solution. For every
$T<\infty$ and every $\alpha\in(0,1)$,
\[
    \omega\in
    L^\infty([0,T];C^{0,1}(\bar\Omega))
    \cap C([0,T];C^{0,\alpha}(\bar\Omega)),
\]
and
\[
    \phi\in
    L^\infty([0,T];C^{2,\alpha}(\bar\Omega)).
\]
Moreover, $\phi\in C([0,T];C^{2,\beta}(\bar\Omega))$ for every
$\beta\in(0,1)$, and for each characteristic $X(t,x)$ the map
$t\mapsto\omega(t,X(t,x))$ is continuously differentiable and satisfies the
characteristic equation in \eqref{eq:system}.

The spatial Lipschitz regularity is propagated by the flow. More precisely,
\begin{equation}
    [\omega(t)]_{\operatorname{Lip}}
    \le
    e^{-t}
    \exp\left(
        \int_0^t\|D^2\phi(s)\|_{L^\infty}\,ds
    \right)
    [\omega_0]_{\operatorname{Lip}},
    \qquad 0\le t\le T.
\label{eq:lipschitz-propagation-summary}
\end{equation}
Finally, if $(\omega_i,\phi_i)$, $i=1,2$, are two such solutions on $[0,T]$,
then
\begin{equation}
    \|\omega_1(t)-\omega_2(t)\|_{L^2(\rho)}^2
    \le
    e^{C_Tt}
    \|\omega_1(0)-\omega_2(0)\|_{L^2(\rho)}^2,
    \qquad 0\le t\le T,
\label{eq:linear-L2-stability-summary}
\end{equation}
where
$\|u\|_{L^2(\rho)}^2=\int_\Omega\rho|u|^2dx$ and $C_T$ depends only on
finite-time bounds of the two solutions and on $\rho$. In particular, the WLF
dynamics are globally well posed in the locally spatially Lipschitz class.
\end{theorem}

The existence argument is first carried out at the H\"older level. Since a
Lipschitz function belongs to $C^{0,\alpha}(\bar\Omega)$ for every
$\alpha\in(0,1)$, this yields a global solution for Lipschitz initial data. A
maximum principle controls the oscillation, while a logarithmic elliptic
estimate obtained from global BMO regularity and a BMO--H\"older interpolation
inequality controls $\|D^2\phi\|_{L^\infty}$ and prevents finite-time blow-up
of the H\"older seminorm. The characteristic formula then propagates spatial Lipschitz
regularity in the form \eqref{eq:lipschitz-propagation-summary}. Finally, a
weighted $L^2(\rho)$ estimate for the Eulerian difference equation, together
with the elliptic $H^1(\rho)$ control of the potential difference, gives
\eqref{eq:linear-L2-stability-summary}, hence uniqueness and continuous
dependence. The detailed proof is given in
\Cref{sec:proof-well-posedness}.

\section{Convergence}
\label{sec:convergence}

The centered relaxation yields an exact contraction of the logarithmic
weights and, consequently, an explicit convergence theory even though the
target is known only through the unnormalized density $\rho$. For this section,
write
\[
    Z_\rho:=\int_\Omega\rho(x)\,dx,
    \qquad
    p:=\frac{\rho}{Z_\rho},
    \qquad
    \eta:=\log\frac{p}{q}.
\]
The normalized density $p$ is used only for the analysis and is not required by
the WLF dynamics.

\begin{theorem}[Exact decay of the log-weight oscillation]
\label{thm:decay}
Let $(\omega,\phi)$ be the unique global solution furnished by \cref{thm:global_wellposedness}. Then
\begin{equation}
    \operatorname{osc}\omega(t)
    :=
    \max_{\bar\Omega}\omega(t,\cdot)-\min_{\bar\Omega}\omega(t,\cdot)
    =
    e^{-t}\operatorname{osc}\omega_0
\label{eq:osc-exact}
\end{equation}
for every $t\ge0$.
\end{theorem}

\begin{proof}
For any $x,y\in\bar\Omega$, subtracting the two characteristic equations gives
\[
\frac{d}{dt}
\left[
\omega(t,X(t,x))-\omega(t,X(t,y))
\right]
=
-
\left[
\omega(t,X(t,x))-\omega(t,X(t,y))
\right].
\]
Hence
\[
\omega(t,X(t,x))-\omega(t,X(t,y))
=
e^{-t}\bigl(\omega_0(x)-\omega_0(y)\bigr).
\]
The characteristic flow is a bijection of $\bar\Omega$ onto itself, so taking
the maximum over $x,y$ yields \eqref{eq:osc-exact}.
\end{proof}

\begin{theorem}[Convergence to the target density]
\label{thm:density-convergence}
Let $q_0$ be a strictly positive probability density such that
\[
    \omega_0:=\log\rho-\log q_0\in C^{0,1}(\bar\Omega),
\]
and let $(\omega,\phi)$ be the corresponding unique global WLF solution from
\cref{thm:global_wellposedness}. Let $q_t$ be transported by
$v=\nabla\phi$, equivalently by \eqref{eq:continuity}. Then, with
\[
    \eta(t,x):=\omega(t,x)-\log Z_\rho
    =\log\frac{p(x)}{q_t(x)},
\]
we have
\begin{equation}
    \|\eta(t,\cdot)\|_{L^\infty(\Omega)}
    \le e^{-t}\operatorname{osc}\omega_0.
\label{eq:eta-exp-bound}
\end{equation}
Consequently,
\[
    \|q_t-p\|_{L^\infty(\Omega)}\longrightarrow0,
    \qquad
    \mathrm{KL}(q_t\,\|\,p)
    \le e^{-t}\operatorname{osc}\omega_0.
\]
\end{theorem}

\begin{proof}
The Neumann boundary condition implies conservation of mass, so
$\int_\Omega q_t\,dx=1$. Since
$q_t=\rho e^{-\omega}=p e^{-\eta}$,
\[
    \int_\Omega p(x)e^{-\eta(t,x)}\,dx=1.
\]
Therefore $\min\eta(t,\cdot)\le0\le\max\eta(t,\cdot)$, and hence
\[
    \|\eta(t,\cdot)\|_{L^\infty}
    \le\operatorname{osc}\eta(t)
    =\operatorname{osc}\omega(t).
\]
Applying \cref{thm:decay} proves \eqref{eq:eta-exp-bound}. The
$L^\infty$ convergence follows from
\[
    q_t-p=p(e^{-\eta}-1),
\]
and the relative entropy satisfies
\[
    0\le \mathrm{KL}(q_t\,\|\,p)
    =-\int_\Omega q_t\eta\,dx
    \le \|\eta\|_{L^\infty}.
\]
\end{proof}

The reverse relative entropy $\mathrm{KL}(p\,\|\,q_t)$ will be connected to
the variational structure in \Cref{sec:variational-geometry}. Here we record a
complementary exact dissipation law for the opposite relative entropy
$\mathrm{KL}(q_t\,\|\,p)$.

\begin{theorem}[Exact entropy dissipation for $\mathrm{KL}(q_t\,\|\,p)$]
\label{thm:entropy-linear}
Under the assumptions of \cref{thm:density-convergence},
\begin{equation}
\label{eq:entropy-dissipation-linear}
    \frac{d}{dt}\mathrm{KL}(q_t\,\|\,p)
    =
    -\mathrm{KL}(q_t\,\|\,p)
    -\mathrm{KL}(p\,\|\,q_t).
\end{equation}
Consequently,
\begin{equation}
\label{eq:entropy-exp-linear}
    \mathrm{KL}(q_t\,\|\,p)
    \le e^{-t}\mathrm{KL}(q_0\,\|\,p),
\end{equation}
and for every $T>0$,
\begin{equation}
\label{eq:entropy-integrated-linear}
    \int_0^T
    \left[
        \mathrm{KL}(q_t\,\|\,p)+\mathrm{KL}(p\,\|\,q_t)
    \right]dt
    =
    \mathrm{KL}(q_0\,\|\,p)-\mathrm{KL}(q_T\,\|\,p).
\end{equation}
Thus the exact entropy dissipation is the Jeffreys divergence.
\end{theorem}

\begin{proof}
Let $\eta=\omega-\log Z_\rho$ and
$\bar\eta=\bar\omega-\log Z_\rho$. Along characteristics,
\[
    \frac{d}{dt}\eta(t,X(t,x))
    =-(\eta(t,X(t,x))-\bar\eta(t)).
\]
By the Neumann compatibility condition,
\[
    \bar\omega(t)=\int_\Omega p\,\omega(t,x)\,dx,
\]
and therefore
\[
    \bar\eta(t)
    :=
    \bar\omega(t)-\log Z_\rho
    =
    \int_\Omega p\,\eta(t,x)\,dx
    =
    \mathrm{KL}(p\,\|\,q_t).
\]
Let
\[
    H(t)=\mathrm{KL}(q_t\,\|\,p)
    =-\int_\Omega q_t\eta\,dx.
\]
Using the pushforward relation $q_t=X(t,\cdot)_\#q_0$ and differentiating along
characteristics,
\[
    H'(t)
    =
    \int_\Omega q_t\eta\,dx-\bar\eta(t)
    =
    -\mathrm{KL}(q_t\,\|\,p)
    -\mathrm{KL}(p\,\|\,q_t).
\]
which proves \eqref{eq:entropy-dissipation-linear}. The remaining conclusions
follow from Gronwall's inequality and integration in time.
\end{proof}

\section{Variational and Geometric Structure}
\label{sec:variational-geometry}

Beyond the exact contraction and entropy identities, the WLF system also
arises from a variational principle on the space of probability densities. 

Let
\begin{equation}
    Z_\rho:=\int_\Omega\rho(x)\,dx,
    \qquad
    p:=\frac{\rho}{Z_\rho},
    \qquad
    \eta:=\log\frac{p}{q}.
\label{eq:geometry-normalization}
\end{equation}
Since $\omega=\eta+\log Z_\rho$, define
\begin{equation}
    \mathcal D_p v
    :=
    \frac1p\nabla\cdot(pv),
    \qquad
    \mathcal L_p\phi
    :=
    \mathcal D_p\nabla\phi,
    \qquad
    P_p g
    :=
    g-\int_\Omega p g\,dx.
\label{eq:weighted-divergence-projection}
\end{equation}

Define the density
space,
\[
    \mathcal P_+(\Omega)
    :=
    \left\{
        q\in C^\infty(\bar\Omega):
        q>0,\ \int_\Omega q\,dx=1
    \right\}.
\]
For $q\in\mathcal P_+(\Omega)$, a tangent vector $\sigma$ satisfies
$\int_\Omega\sigma\,dx=0$. Let $\psi_\sigma$ solve
\begin{equation}
    \sigma=-\nabla\cdot(q\nabla\psi_\sigma),
    \qquad
    \partial_n\psi_\sigma=0,
    \qquad
    \int_\Omega\psi_\sigma\,dx=0.
\label{eq:tangent-potential}
\end{equation}
We define the target-weighted divergence metric
\begin{equation}
    G_q^{\mathrm{WLF}}(\sigma_1,\sigma_2)
    :=
    \int_\Omega
    p\,
    (\mathcal L_p\psi_{\sigma_1})
    (\mathcal L_p\psi_{\sigma_2})\,dx.
\label{eq:wlf-metric}
\end{equation}


\begin{proposition}[Gradient-flow structure]
\label{prop:reverse-kl-gradient-flow}
Let
\[
    \mathcal E(q):=\mathrm{KL}(p\,\|\,q).
\]
For smooth positive densities, WLF is the formal gradient flow
\begin{equation}
    \partial_t q
    =
    -\operatorname{grad}_{G^{\mathrm{WLF}}}\mathcal E(q).
\label{eq:wlf-gradient-flow}
\end{equation}
Along every smooth WLF solution,
\begin{equation}
    \frac{d}{dt}\mathrm{KL}(p\,\|\,q_t)
    =
    -
    \int_\Omega
    p\left|
        P_p\log\frac{p}{q_t}
    \right|^2dx
    =
    -
    \operatorname{Var}_p\!\left(\log\frac{p}{q_t}\right).
\label{eq:reverse-kl-dissipation}
\end{equation}
where
\[
    \operatorname{Var}_p(g)
    :=
    \int_\Omega p
    \left|g-\int_\Omega pg\,dx\right|^2dx.
\]
Equivalently, for every $T>0$,
\begin{equation}
    \mathrm{KL}(p\,\|\,q_T)
    +
    \int_0^T
    \operatorname{Var}_p\!\left(\log\frac{p}{q_t}\right)\,dt
    =
    \mathrm{KL}(p\,\|\,q_0).
\label{eq:reverse-kl-energy-identity}
\end{equation}
\end{proposition}

\begin{proof}
Let $\sigma=-\nabla\cdot(q\nabla\psi)$ be an arbitrary tangent vector as in
\eqref{eq:tangent-potential}. The first variation of the $\mathrm{KL}(p\,\|\,q)$ is
\begin{align*}
    D\mathcal E(q)[\sigma]
    &=
    -\int_\Omega p\,\frac{\sigma}{q}\,dx\\
    &=
    \int_\Omega\frac{p}{q}
        \nabla\cdot(q\nabla\psi)\,dx\\
    &=
    -\int_\Omega p
        \nabla\log\frac{p}{q}\cdot\nabla\psi\,dx\\
    &=
    \int_\Omega p
        \log\frac{p}{q}\,
        \mathcal L_p\psi\,dx.
\end{align*}
Because
\[
    \int_\Omega p\,\mathcal L_p\psi\,dx=0,
\]
the last expression equals
\[
    \int_\Omega p
        P_p\!\left(\log\frac{p}{q}\right)
        \mathcal L_p\psi\,dx.
\]
Let $u$ be the normalized Neumann solution of
\[
    \mathcal L_pu
    =
    P_p\!\left(\log\frac{p}{q}\right).
\]
Then, by \eqref{eq:wlf-metric},
\[
    D\mathcal E(q)[\sigma]
    =
    G_q^{\mathrm{WLF}}
    \bigl(-\nabla\cdot(q\nabla u),\sigma\bigr).
\]
Thus
\[
    \operatorname{grad}_{G^{\mathrm{WLF}}}\mathcal E(q)
    =
    -\nabla\cdot(q\nabla u).
\]
The WLF potential satisfies
$\mathcal L_p\phi=-P_p\log(p/q)$, so $\phi=-u$ up to the fixed additive
normalization. Hence the continuity equation
$\partial_tq=-\nabla\cdot(q\nabla\phi)$ is precisely
\eqref{eq:wlf-gradient-flow}.

Finally,
\[
\begin{aligned}
    \frac{d}{dt}\mathcal E(q_t)
    &=
    -G_{q_t}^{\mathrm{WLF}}
    \left(
        \operatorname{grad}\mathcal E,
        \operatorname{grad}\mathcal E
    \right)\\
    &=
    -\int_\Omega
    p\left|
        P_p\log\frac{p}{q_t}
    \right|^2dx,
\end{aligned}
\]
which proves \eqref{eq:reverse-kl-dissipation}. Integration in time gives
\eqref{eq:reverse-kl-energy-identity}.
\end{proof}

\begin{remark}[Comparison with Wasserstein/Langevin geometry]
\label{rem:wasserstein-comparison}
The overdamped Langevin equation is the Wasserstein gradient flow of the
forward divergence $\mathrm{KL}(q\,\|\,p)$, and its deterministic
probability-flow velocity is
\[
    v_{\mathrm{Lan}}=\nabla\eta.
\]
By contrast, if $A_p:=-\mathcal L_p$ on the $p$-mean-zero subspace, then
\[
    v_{\mathrm{WLF}}
    =
    \nabla A_p^{-1}P_p\eta.
\]
Thus WLF replaces the local score force by an inverse weighted-elliptic
preconditioning before taking the spatial gradient. This emphasizes the global
elliptic character of the WLF transport.
\end{remark}

\begin{remark}[Local geometry at equilibrium]
\label{rem:local-wlf-geometry}
The WLF metric has a particularly simple form at $q=p$. If a tangent vector is
written as $\sigma=p\xi$ with $\int_\Omega p\xi\,dx=0$, then
\eqref{eq:tangent-potential} gives
$\xi=-\mathcal L_p\psi_\sigma$, and hence
\[
    G_p^{\mathrm{WLF}}(\sigma,\sigma)
    =
    \int_\Omega p\xi^2\,dx.
\]
On the other hand,
\[
    \mathrm{KL}\bigl(p\,\|\,p(1+\varepsilon\xi)\bigr)
    =
    \frac{\varepsilon^2}{2}
    \int_\Omega p\xi^2\,dx
    +
    O(\varepsilon^3).
\]
Thus the metric tensor agrees with the Hessian of the $\mathrm{KL}(p\,\|\,q)$ at equilibrium.
Equivalently, if
$q=p(1+\varepsilon u)$ with $\int_\Omega pu\,dx=0$, then the linearized WLF is
\[
    \partial_tu=-u.
\]
The unit relaxation rate therefore arises directly from the local variational
geometry, rather than from a spectral-gap estimate for $\mathcal L_p$.
\end{remark}

\section{Proof of Global Well-Posedness}
\label{sec:proof-well-posedness}

This section supplies the detailed proof of \cref{thm:global_wellposedness}.
We first construct a global solution at the H\"older level. Since Lipschitz
initial data belong to $C^{0,\alpha}(\bar\Omega)$ for every
$\alpha\in(0,1)$, this provides existence in the class required by
\cref{thm:global_wellposedness}. We then propagate spatial Lipschitz
regularity and close the weighted $L^2(\rho)$ stability estimate, which gives
uniqueness and continuous dependence.

\subsection{Local existence by fixed point theorem}
\label{subsec:local-existence-proof}

We first make the fixed-point construction precise. Let
$0<\alpha<1$ and let $T>0$. For $M>0$ set
\begin{equation}
\label{eq:S}
\mathcal S_M
:=
\left\{
\begin{array}{l|l}
\phi\in C([0,T];C(\bar\Omega))
&
\displaystyle
\sup_{t\in[0,T]}\|\phi(t)\|_{C^{2,\alpha}(\bar\Omega)}\le M,
\\[-1mm]
&
\partial_n\phi(t)=0,\quad
\displaystyle\int_\Omega\phi(t,x)\,dx=0
\end{array}
\right\}.
\end{equation}
The set $\mathcal S_M$ is convex and closed in
$C([0,T];C(\bar\Omega))$. Notice also that the uniform
$C^{2,\alpha}$ bound and interpolation imply
\[
    \mathcal S_M\subset C([0,T];C^{2,\beta}(\bar\Omega))
    \qquad\text{for every }\beta<\alpha.
\]

For $\phi\in\mathcal S_M$, let $X$ be the flow generated by
$v=\nabla\phi$ and define $\omega=T_1(\phi)$ by
\begin{equation}
\label{eq:transport}
\begin{cases}
\displaystyle
\frac{d}{dt}\omega(t,X(t,a))
=
-\bigl(\omega(t,X(t,a))-\bar\omega(t)\bigr),\\[1mm]
\omega(0,\cdot)=\omega_0,
\end{cases}
\qquad
\bar\omega(t)
=
\frac{\int_\Omega\rho\,\omega(t)\,dx}
     {\int_\Omega\rho\,dx}.
\end{equation}
For a fixed flow this is a linear nonlocal ordinary differential equation after
pullback by $X$, hence it has a unique solution in
$C([0,T];C(\bar\Omega))$. Define $T_2$ at each time by the normalized weighted
Neumann problem
\begin{equation}
\label{eq:elliptic}
\begin{cases}
\displaystyle
\frac{1}{\rho}\nabla\cdot(\rho\nabla\phi)
=
-(\omega-\bar\omega)
&\text{in }\Omega,\\
\partial_n\phi=0
&\text{on }\partial\Omega,\\
\displaystyle\int_\Omega\phi\,dx=0.
\end{cases}
\end{equation}
Thus $\mathcal T:=T_2\circ T_1$ maps $\mathcal S_M$ into
$C([0,T];C(\bar\Omega))$ whenever the right-hand side is defined.

We use the standard Schauder estimate for the normalized Neumann problem; see,
for example, Chapter~6 of \cite{Gilbarg1977EllipticPD}.

\begin{proposition}[Schauder estimate]
\label{prop:schauder_estimate}
For every $\alpha\in(0,1)$ there is a constant
$C_{\alpha,\Omega,\rho}>0$ such that
\[
    \|T_2(\omega)(t)\|_{C^{2,\alpha}(\bar\Omega)}
    \le
    C_{\alpha,\Omega,\rho}
    \|\omega(t)-\bar\omega(t)\|_{C^{0,\alpha}(\bar\Omega)}
\]
for every $t$.
\end{proposition}

\begin{lemma}
\label{lem:norm_transport}
Suppose that
$\phi\in C([0,T];C^{2,\beta}(\bar\Omega))$ for some $\beta>0$, with
$\sup_t\|\phi(t)\|_{C^2}<\infty$, and let $\omega=T_1(\phi)$. Then, for every
$\alpha\in(0,1]$ for which $\omega_0\in C^{0,\alpha}(\bar\Omega)$,
\begin{equation}
    [\omega(t)]_{C^{0,\alpha}}
    \le
    e^{-t}
    \exp\left(
        \alpha\int_0^t
        \|D^2\phi(s)\|_{L^\infty}\,ds
    \right)
    [\omega_0]_{C^{0,\alpha}},
\label{eq:linear-holder-transport}
\end{equation}
and
\[
    \|\omega(t)\|_{L^\infty}
    \le
    \|\omega_0\|_{L^\infty}.
\]
Moreover, for $0\le t'<t\le T$,
\begin{equation}
    \|\omega(t)-\omega(t')\|_{L^\infty}
    \le
    C_T\bigl(|t-t'|+|t-t'|^\alpha\bigr),
\label{eq:linear-time-modulus}
\end{equation}
where $C_T$ depends only on
$\|\omega_0\|_{C^{0,\alpha}}$ and the indicated finite-time bound for $\phi$.
\end{lemma}

\begin{proof}
Because $\bar\omega(t)$ is a positive weighted average of $\omega(t,\cdot)$,
the maximum of $\omega$ is nonincreasing and the minimum is nondecreasing.
This proves the $L^\infty$ bound. The flow satisfies
\[
    |X(t,x)-X(t,y)|
    \ge
    |x-y|
    \exp\left(
        -\int_0^t\|D^2\phi(s)\|_{L^\infty}\,ds
    \right).
\]
On the other hand, subtracting the two characteristic equations gives the
exact identity
\[
    \omega(t,X(t,x))-\omega(t,X(t,y))
    =
    e^{-t}\bigl(\omega_0(x)-\omega_0(y)\bigr).
\]
Combining the preceding two formulas proves
\eqref{eq:linear-holder-transport}.

For the time modulus, choose $a$ so that $X(t',a)=x$. The characteristic
equation and the uniform $L^\infty$ bound give
\[
    |\omega(t,X(t,a))-\omega(t',x)|
    \le C|t-t'|.
\]
Furthermore,
\[
    |X(t,a)-x|
    \le
    |t-t'|\|\nabla\phi\|_{C([0,T];L^\infty)}.
\]
Applying the spatial H\"older estimate at time $t$ and then the triangle
inequality yields \eqref{eq:linear-time-modulus}.
\end{proof}

\begin{lemma}[Continuity of the fixed-point map]
\label{lem:fixed-point-continuity}
The map
\[
    \mathcal T=T_2\circ T_1:
    \mathcal S_M\longrightarrow C([0,T];C(\bar\Omega))
\]
is continuous.
\end{lemma}

\begin{proof}
Let $\phi_n,\phi\in\mathcal S_M$ and assume
$\phi_n\to\phi$ in $C([0,T];C(\bar\Omega))$. By interpolation and the uniform
$C^{2,\alpha}$ bound, for every $\beta<\alpha$,
\[
    \phi_n\longrightarrow\phi
    \quad\text{in }C([0,T];C^{2,\beta}(\bar\Omega)).
\]
Hence the vector fields $v_n=\nabla\phi_n$ and their first spatial derivatives
converge uniformly. Standard stability of ordinary differential equations then
gives uniform convergence of the flows $X_n$, their inverses, and their
Jacobians $J_n=\det DX_n$ on $[0,T]\times\bar\Omega$.

Let $\omega_n=T_1(\phi_n)$ and put
$w_n(t,a)=\omega_n(t,X_n(t,a))$. After the change of variables
$x=X_n(t,a)$, \eqref{eq:transport} becomes the linear Banach-space ODE
\[
    \partial_t w_n(t,a)
    =
    -w_n(t,a)+\ell_n(t)[w_n(t,\cdot)],
\]
where
\[
    \ell_n(t)[g]
    :=
    \frac{\int_\Omega
        \rho(X_n(t,a))\,g(a)\,J_n(t,a)\,da}
         {\int_\Omega\rho\,dx}.
\]
The convergence of $X_n$ and $J_n$ implies
$\ell_n\to\ell$ uniformly in operator norm on $C(\bar\Omega)^*$.
Continuous dependence for this linear ODE therefore yields
\[
    w_n\longrightarrow w
    \quad\text{in }C([0,T];C(\bar\Omega)).
\]
Together with $X_n^{-1}\to X^{-1}$, this gives
$\omega_n\to\omega$ in $C([0,T];C(\bar\Omega))$ and consequently
$\bar\omega_n\to\bar\omega$ uniformly in time.

Finally, the normalized weighted Neumann inverse is continuous from
$C(\bar\Omega)$ into $C(\bar\Omega)$; for instance this follows from the
standard $W^{2,p}$ estimate with any $p>d$ and Sobolev embedding. Applying it
to
$-(\omega_n-\bar\omega_n)+(\omega-\bar\omega)$ proves
$\mathcal T\phi_n\to\mathcal T\phi$ in
$C([0,T];C(\bar\Omega))$.
\end{proof}

\begin{lemma}
\label{lem:compact}
The map
$\mathcal T:\mathcal S_M\to C([0,T];C(\bar\Omega))$
has relatively compact image.
\end{lemma}

\begin{proof}
By \cref{lem:norm_transport}, the family
$\{T_1(\phi):\phi\in\mathcal S_M\}$ is uniformly bounded in
$L^\infty([0,T];C^{0,\alpha}(\bar\Omega))$ and has a common time modulus in
$C(\bar\Omega)$. The Schauder estimate gives a uniform spatial
$C^{2,\alpha}$ bound for $\mathcal T(\mathcal S_M)$. The normalized weighted
Neumann inverse, together with \eqref{eq:linear-time-modulus}, gives a common
time modulus for $\mathcal T(\mathcal S_M)$ in $C(\bar\Omega)$.
The compact embedding
$C^{2,\alpha}(\bar\Omega)\Subset C(\bar\Omega)$ and the
Arzel\`a--Ascoli theorem therefore imply relative compactness in
$C([0,T];C(\bar\Omega))$.
\end{proof}

\begin{theorem}[Local-in-time existence]
\label{thm:local_in_time_existence}
Suppose that
$\omega_0\in C^{0,\alpha}(\bar\Omega)$ for some $\alpha\in(0,1)$.
There is a constant $C=C(\Omega,\rho,\alpha)>0$ such that for every
$\gamma>1$ and every $T$ satisfying
\begin{equation}
    T
    \le
    \frac{\log\gamma}
         {\alpha\gamma C
          (1+\|\omega_0\|_{C^{0,\alpha}(\bar\Omega)})},
\label{eq:local-lifespan}
\end{equation}
there exists a solution
\[
    \omega\in
    L^\infty([0,T];C^{0,\alpha}(\bar\Omega))
    \cap C([0,T];C(\bar\Omega)),
\]
\[
    \phi\in
    L^\infty([0,T];C^{2,\alpha}(\bar\Omega))
    \cap C([0,T];C(\bar\Omega))
\]
of \eqref{eq:system}. For every $\beta<\alpha$, interpolation gives
$\omega\in C([0,T];C^{0,\beta})$ and
$\phi\in C([0,T];C^{2,\beta})$. Along each characteristic,
$t\mapsto\omega(t,X(t,a))$ is continuously differentiable.
\end{theorem}

\begin{proof}
Choose, after enlarging the Schauder constant if necessary,
\[
    M
    :=
    \gamma C
    \bigl(1+\|\omega_0\|_{C^{0,\alpha}}\bigr).
\]
For $\phi\in\mathcal S_M$, \cref{lem:norm_transport} gives
\[
    \sup_{t\le T}\|T_1(\phi)(t)\|_{C^{0,\alpha}}
    \le
    \|\omega_0\|_{L^\infty}
    +
    e^{\alpha MT}[\omega_0]_{C^{0,\alpha}}
    \le
    1+e^{\alpha MT}\|\omega_0\|_{C^{0,\alpha}},
\]
where the harmless constant $1$ is absorbed into $C$. Consequently,
\[
    \sup_{t\le T}
    \|\mathcal T\phi(t)\|_{C^{2,\alpha}}
    \le
    C\left(
        1+e^{\alpha MT}
        \|\omega_0\|_{C^{0,\alpha}}
    \right).
\]
Condition \eqref{eq:local-lifespan} implies
$e^{\alpha MT}\le\gamma$, and hence
\[
    C\left(
        1+\gamma\|\omega_0\|_{C^{0,\alpha}}
    \right)
    \le
    \gamma C
    \left(1+\|\omega_0\|_{C^{0,\alpha}}\right)
    =M.
\]
Thus $\mathcal T(\mathcal S_M)\subset\mathcal S_M$.
By \cref{lem:fixed-point-continuity,lem:compact},
$\mathcal T$ is a continuous compact self-map of the nonempty closed convex
set $\mathcal S_M$. Schauder's fixed-point theorem gives a fixed point.
The regularity and the characteristic formulation follow from the construction
and \cref{lem:norm_transport}.
\end{proof}

\subsection{Global continuation}
\label{subsec:global-continuation-proof}

We next derive the logarithmic elliptic estimate needed to prevent finite-time
growth of the H\"older norm. The argument uses two standard endpoint estimates
for smooth regular elliptic Neumann problems: global BMO regularity and the
Schauder estimate. For convenience, write
\[
    \|g\|_{\mathrm{BMO}_*(\Omega)}
    :=
    \|g\|_{L^1(\Omega)}
    +
    [g]_{\mathrm{BMO}(\Omega)}.
\]
For smooth bounded domains, the global BMO estimate for regular elliptic
boundary-value problems gives, for the normalized solution of
\[
    \mathcal L_\rho\phi
    :=
    \frac1\rho\nabla\cdot(\rho\nabla\phi)
    =f,
    \qquad
    \partial_n\phi=0,
    \qquad
    \int_\Omega\phi\,dx=0,
\]
the bound
\begin{equation}
    \|D^2\phi\|_{\mathrm{BMO}_*(\Omega)}
    \le
    C_{\Omega,\rho}\|f\|_{L^\infty(\Omega)}.
\label{eq:bmo-elliptic}
\end{equation}
This is a standard endpoint regularity estimate for regular elliptic boundary
problems; see, for example, \cite{elbaraka2005bmo}. The lower-order coefficient
$\nabla\log\rho$ is smooth, and the homogeneous Neumann condition satisfies the
usual complementing condition.

We also record the elementary logarithmic interpolation step used below.

\begin{lemma}[BMO--H\"older logarithmic interpolation]
\label{lem:bmo-holder-log}
Let $\Omega$ be a bounded smooth domain and let $\alpha\in(0,1)$. There is a
constant $C=C(\Omega,\alpha)$ such that every
$g\in C^{0,\alpha}(\bar\Omega)$ satisfies
\begin{equation}
    \|g\|_{L^\infty(\Omega)}
    \le
    C M
    \left[
        1+
        \log^+\left(\frac{H}{M}\right)
    \right],
\label{eq:bmo-holder-log}
\end{equation}
where
\[
    M:=\|g\|_{\mathrm{BMO}_*(\Omega)},
    \qquad
    H:=\|g\|_{C^{0,\alpha}(\bar\Omega)},
\]
with the right-hand side interpreted as zero when $M=0$.
\end{lemma}

\begin{proof}
Because $\Omega$ is smooth, there are $r_0,c_0>0$ such that
$|\Omega\cap B_r(x)|\ge c_0r^d$ for every $x\in\bar\Omega$ and
$0<r\le r_0$. Denote
$\Omega_r(x):=\Omega\cap B_r(x)$ and by $g_{\Omega_r(x)}$ the average of $g$
over this set. H\"older continuity gives
\[
    |g(x)-g_{\Omega_r(x)}|
    \le
    C H r^\alpha.
\]
Choose an integer $N\ge0$ so that
$2^Nr\le r_0<2^{N+1}r$. By telescoping the averages over the nested sets
$\Omega_{2^kr}(x)$ and using the uniform volume comparability furnished by the
smooth boundary,
\[
    |g_{\Omega_r(x)}-g_{\Omega_{2^Nr}(x)}|
    \le
    C N [g]_{\mathrm{BMO}(\Omega)}.
\]
The average on the largest set is bounded by
$C\|g\|_{L^1(\Omega)}$. Hence
\[
    |g(x)|
    \le
    C\bigl(Hr^\alpha+M(1+\log(r_0/r))\bigr).
\]
If $H\le Mr_0^{-\alpha}$, take $r=r_0$. Otherwise take
$r=(M/H)^{1/\alpha}<r_0$. This yields
\eqref{eq:bmo-holder-log}, uniformly in $x$.
\end{proof}

\begin{lemma}[Logarithmic elliptic control]
\label{lem:log_control}
Fix $\alpha\in(0,1)$ and $\kappa>0$. Let
$\omega\in C^{0,\alpha}(\bar\Omega)$ satisfy
$\|\omega\|_{L^\infty}\le\kappa$, let
\[
    \bar\omega
    =
    \frac{\int_\Omega\rho\omega\,dx}{\int_\Omega\rho\,dx},
\]
and let $\phi$ solve
\[
    \mathcal L_\rho\phi
    =
    -(\omega-\bar\omega),
    \qquad
    \partial_n\phi=0,
    \qquad
    \int_\Omega\phi\,dx=0.
\]
Then
\begin{equation}
    \|D^2\phi\|_{L^\infty(\Omega)}
    \le
    C_{\Omega,\rho,\alpha,\kappa}
    \left[
        1+\log\!\left(e+\|\omega\|_{C^{0,\alpha}(\bar\Omega)}\right)
    \right].
\label{eq:log-elliptic-control}
\end{equation}
\end{lemma}

\begin{proof}
Set $f:=-(\omega-\bar\omega)$. Since $\bar\omega$ is a weighted average,
\[
    \|f\|_{L^\infty}\le 2\kappa,
    \qquad
    \|f\|_{C^{0,\alpha}}
    \le
    C_\kappa\left(1+\|\omega\|_{C^{0,\alpha}}\right).
\]
By \eqref{eq:bmo-elliptic},
\[
    \|D^2\phi\|_{\mathrm{BMO}_*}
    \le
    C_{\Omega,\rho}\|f\|_{L^\infty}.
\]
The global Schauder estimate and the normalization of $\phi$ give
\[
    \|D^2\phi\|_{C^{0,\alpha}}
    \le
    C_{\Omega,\rho,\alpha}
    \left(
        \|f\|_{C^{0,\alpha}}+\|\phi\|_{L^\infty}
    \right)
    \le
    C_{\Omega,\rho,\alpha,\kappa}
    \left(1+\|\omega\|_{C^{0,\alpha}}\right),
\]
where the last bound follows from a standard normalized Neumann
$W^{2,p}$ estimate with $p>d$. Applying
\cref{lem:bmo-holder-log} to $g=D^2\phi$ proves
\eqref{eq:log-elliptic-control}. The case
$\|D^2\phi\|_{\mathrm{BMO}_*}=0$ is immediate.
\end{proof}

\begin{theorem}[Global existence for H\"older data]
\label{thm:global_existence}
Suppose that
$\omega_0\in C^{0,\alpha}(\bar\Omega)$ for some $\alpha\in(0,1)$.
Then there exist
\[
    \omega\in
    L^\infty_{\mathrm{loc}}([0,\infty);C^{0,\alpha}(\bar\Omega))
    \cap C([0,\infty);C(\bar\Omega)),
\]
\[
    \phi\in
    L^\infty_{\mathrm{loc}}([0,\infty);C^{2,\alpha}(\bar\Omega))
    \cap C([0,\infty);C(\bar\Omega))
\]
solving \eqref{eq:system}, with $C^1$ evolution of $\omega$ along
characteristics. For every $\beta<\alpha$, both functions are continuous in
time with values in $C^{0,\beta}$ and $C^{2,\beta}$, respectively.
\end{theorem}

\begin{proof}
Start with the local solution furnished by
\cref{thm:local_in_time_existence}. On every interval on which the solution
has been constructed, \cref{lem:norm_transport} gives
\[
    \|\omega(t)\|_{L^\infty}
    \le
    \|\omega_0\|_{L^\infty}
    =:\kappa
\]
and
\begin{equation}
    [\omega(t)]_{C^{0,\alpha}}
    \le
    [\omega_0]_{C^{0,\alpha}}
    \exp\left(
        \alpha\int_0^t
        \|D^2\phi(s)\|_{L^\infty}\,ds
    \right).
\label{eq:holder-growth-global}
\end{equation}
Define
\[
    A(t):=
    \int_0^t\|D^2\phi(s)\|_{L^\infty}\,ds.
\]
Using \cref{lem:log_control} and \eqref{eq:holder-growth-global}, for almost
every time at which the solution is defined,
\[
\begin{aligned}
    A'(t)
    &\le
    C_{\kappa}
    \left[
        1+
        \log\left(
            e+
            \|\omega(t)\|_{C^{0,\alpha}}
        \right)
    \right]\\
    &\le
    C_{\kappa,\omega_0}
    \bigl(1+A(t)\bigr),
\end{aligned}
\]
where the constants also depend on the fixed data
$(\Omega,\rho,\alpha)$. Gronwall's inequality therefore yields, for every
finite $T>0$, an a priori bound
\begin{equation}
    \sup_{0\le t\le T}
    \|\omega(t)\|_{C^{0,\alpha}}
    \le
    K(T)<\infty
\label{eq:finite-holder-bound}
\end{equation}
on every solution segment contained in $[0,T]$, where $K(T)$ depends only on
$T$, $\|\omega_0\|_{C^{0,\alpha}}$, and the fixed data, not on the particular
continuation.

We now continue the solution iteratively, which avoids using uniqueness at the
H\"older level. Suppose that a solution has been constructed up to a time
$t_0<T$. Apply \cref{thm:local_in_time_existence} with initial datum
$\omega(t_0)$ and choose $\gamma=2$. By
\eqref{eq:finite-holder-bound}, the new local solution exists for at least
\[
    \tau_T
    :=
    \frac{\log2}
         {2\alpha C(1+K(T))}>0,
\]
as long as $t_0<T$. Concatenating at $t_0$ extends the previously constructed
solution, and the same a priori estimate remains valid on the enlarged
interval. Therefore at most
$\lceil T/\tau_T\rceil+1$ such continuation steps cover $[0,T]$.
Since $T>0$ is arbitrary, repeating this construction on successive finite
time horizons produces a global solution on $[0,\infty)$.
\end{proof}

\subsection{Uniqueness and stability}
\label{subsec:uniqueness-proof}

We first record that the additional spatial regularity needed by the stability
argument is propagated by the flow.

\begin{lemma}[Propagation of spatial Lipschitz regularity]
\label{lem:lipschitz-propagation}
Let $(\omega,\phi)$ be a global solution of WLF furnished by
\cref{thm:global_existence} and assume
$\omega_0\in C^{0,1}(\bar\Omega)$. Then, for every $T<\infty$,
\[
    \omega\in L^\infty([0,T];C^{0,1}(\bar\Omega)).
\]
More precisely,
\begin{equation}
    [\omega(t)]_{\operatorname{Lip}}
    \le
    e^{-t}
    \exp\left(
        \int_0^t
        \|D^2\phi(s)\|_{L^\infty}\,ds
    \right)
    [\omega_0]_{\operatorname{Lip}},
    \qquad 0\le t\le T.
\label{eq:lipschitz-propagation}
\end{equation}
\end{lemma}

\begin{proof}
Let $X_t(x):=X(t,x)$ denote the characteristic flow. Since
$\nabla\phi\in L^\infty([0,T];C^{1,\alpha}(\bar\Omega))$, the map $X_t$ is
bi-Lipschitz and
\begin{equation}
    \operatorname{Lip}(X_t^{-1})
    \le
    \exp\left(
        \int_0^t\|D^2\phi(s)\|_{L^\infty}\,ds
    \right).
\label{eq:inverse-flow-lipschitz}
\end{equation}
For any $x,y\in\bar\Omega$, subtracting the two characteristic equations gives
\[
    \omega(t,X_t(x))-\omega(t,X_t(y))
    =
    e^{-t}\bigl(\omega_0(x)-\omega_0(y)\bigr).
\]
Given $z_1,z_2\in\bar\Omega$, set
$x=X_t^{-1}(z_1)$ and $y=X_t^{-1}(z_2)$. Then
\[
\begin{aligned}
    |\omega(t,z_1)-\omega(t,z_2)|
    &\le
    e^{-t}[\omega_0]_{\operatorname{Lip}}|x-y|\\
    &\le
    e^{-t}
    \exp\left(
        \int_0^t\|D^2\phi(s)\|_{L^\infty}\,ds
    \right)
    [\omega_0]_{\operatorname{Lip}}|z_1-z_2|.
\end{aligned}
\]
This proves \eqref{eq:lipschitz-propagation}. The integral in the exponent is
finite on every bounded time interval by the global continuation argument.
\end{proof}

\begin{theorem}[Weighted $L^2$ stability and uniqueness]
\label{thm:uniqueness}
Let $(\omega_i,\phi_i)$, $i=1,2$, be two solutions of
\eqref{eq:system} on $[0,T]$ such that, for some $\alpha\in(0,1)$,
\[
    \omega_i\in L^\infty([0,T];W^{1,\infty}(\Omega)),
    \qquad
    \phi_i\in L^\infty([0,T];C^{2,\alpha}(\bar\Omega))
    \cap C([0,T];C(\bar\Omega)).
\]
Then
\begin{equation}
    \|\omega_1(t)-\omega_2(t)\|_{L^2(\rho)}^2
    \le
    e^{C_Tt}
    \|\omega_1(0)-\omega_2(0)\|_{L^2(\rho)}^2,
\label{eq:linear-L2-stability}
\end{equation}
where $C_T$ depends on $\rho$ and on the indicated finite-time norms of the
solutions. In particular, two such solutions with the same initial data
coincide.
\end{theorem}

\begin{proof}
Set
\[
    \delta:=\omega_1-\omega_2,
    \qquad
    \psi:=\phi_1-\phi_2,
    \qquad
    \bar\delta
    :=
    \frac{\int_\Omega\rho\delta\,dx}{\int_\Omega\rho\,dx}.
\]
Writing
\[
    F_i:=-(\omega_i-\bar\omega_i),
    \qquad
    v_i:=\nabla\phi_i,
\]
we first justify the Eulerian form of the transport equation. The Neumann
compatibility condition gives
\[
    \bar\omega_i(t)
    =
    \frac{\int_\Omega\rho\,\omega_i(t)\,dx}
         {\int_\Omega\rho\,dx},
\]
and therefore $F_i\in L^\infty((0,T)\times\Omega)$. Let $X_i$ be the
characteristic flow generated by $v_i$. Since
$v_i\in L^\infty(0,T;C^{1,\alpha}(\bar\Omega))$, the maps
$X_i(t,\cdot)$ are bi-Lipschitz on every finite time interval. Fix
$0\le s<t\le T$ and $x\in\bar\Omega$, choose $a$ such that
$X_i(s,a)=x$, and set $y:=X_i(t,a)$. The characteristic equation and the
spatial Lipschitz bound imply
\[
\begin{aligned}
    |\omega_i(t,x)-\omega_i(s,x)|
    &\le
    |\omega_i(t,x)-\omega_i(t,y)|
    +
    |\omega_i(t,y)-\omega_i(s,x)|
    \\
    &\le
    \|\nabla\omega_i\|_{L^\infty((0,T)\times\Omega)}
    |x-y|
    +
    \int_s^t
    |F_i(r,X_i(r,a))|\,dr
    \\
    &\le
    \left(
        \|\nabla\omega_i\|_{L^\infty((0,T)\times\Omega)}
        \|v_i\|_{L^\infty((0,T)\times\Omega)}
        +
        \|F_i\|_{L^\infty((0,T)\times\Omega)}
    \right)|t-s|.
\end{aligned}
\]
Thus $\omega_i$ is Lipschitz also in time, uniformly in space. Together with
the assumed spatial Lipschitz regularity, this yields
\[
    \omega_i\in W^{1,\infty}((0,T)\times\Omega).
\]
To pass rigorously from the characteristic formulation to the Eulerian
equation, consider the spacetime map
\[
    \Phi_i:(t,a)\longmapsto (t,X_i(t,a)).
\]
On every finite time interval $\Phi_i$ is bi-Lipschitz. Since
$\omega_i\in W^{1,\infty}((0,T)\times\Omega)$, Rademacher's theorem and the
Sobolev chain rule imply, for almost every $(t,a)$,
\[
    \frac{d}{dt}\omega_i(t,X_i(t,a))
    =
    \partial_t\omega_i(t,X_i(t,a))
    +
    v_i(t,X_i(t,a))\cdot
    \nabla\omega_i(t,X_i(t,a)).
\]
Comparing with the characteristic equation and using that the bi-Lipschitz map
$\Phi_i$ sends null sets to null sets in both directions, we obtain
\[
    \partial_t\omega_i+v_i\cdot\nabla\omega_i=F_i
\]
almost everywhere in $(0,T)\times\Omega$.

Consequently, the difference equations hold almost everywhere:
\begin{align}
    \partial_t\delta
    +v_1\cdot\nabla\delta
    +\nabla\psi\cdot\nabla\omega_2
    &=
    -(\delta-\bar\delta),
\label{eq:linear-difference-eulerian}
\\
    \frac1\rho\nabla\cdot(\rho\nabla\psi)
    &=
    -(\delta-\bar\delta).
\label{eq:linear-difference-elliptic}
\end{align}
For the elliptic equation, define the weighted mean
\[
    \langle\psi\rangle_\rho
    :=
    \frac{\int_\Omega\rho\psi\,dx}{\int_\Omega\rho\,dx}.
\]
Testing \eqref{eq:linear-difference-elliptic} with $\psi$ and using
$\int_\Omega\rho(\delta-\bar\delta)\,dx=0$ gives
\[
\begin{aligned}
    \|\nabla\psi\|_{L^2(\rho)}^2
    &=
    \int_\Omega
    \rho(\delta-\bar\delta)
    (\psi-\langle\psi\rangle_\rho)\,dx
    \\
    &\le
    \|\delta-\bar\delta\|_{L^2(\rho)}
    \|\psi-\langle\psi\rangle_\rho\|_{L^2(\rho)}
    \\
    &\le
    C_\rho
    \|\delta-\bar\delta\|_{L^2(\rho)}
    \|\nabla\psi\|_{L^2(\rho)}.
\end{aligned}
\]
Hence
\begin{equation}
    \|\nabla\psi\|_{L^2(\rho)}
    \le
    C_\rho\|\delta-\bar\delta\|_{L^2(\rho)}
    \le
    C_\rho\|\delta\|_{L^2(\rho)},
\label{eq:linear-H1-elliptic-difference}
\end{equation}
where the last inequality follows because $\bar\delta$ is the
$L^2(\rho)$-orthogonal projection of $\delta$ onto the constants.

Since $\delta\in W^{1,\infty}((0,T)\times\Omega)$, the map
$t\mapsto\|\delta(t)\|_{L^2(\rho)}^2$ is absolutely continuous and, for
almost every $t$,
\[
    \frac12\frac{d}{dt}\|\delta(t)\|_{L^2(\rho)}^2
    =
    \int_\Omega\rho\delta\,\partial_t\delta\,dx.
\]
Thus, for almost every $t$, we may multiply
\eqref{eq:linear-difference-eulerian} by $\rho\delta$ and integrate. Since
\[
    \nabla\cdot(\rho v_1)=\rho F_1
\]
and $v_1\cdot n=0$, integration by parts yields
\begin{align*}
\frac12\frac{d}{dt}\|\delta\|_{L^2(\rho)}^2
&=
\frac12\int_\Omega\rho F_1\delta^2\,dx
-
\int_\Omega\rho\delta\,\nabla\psi\cdot\nabla\omega_2\,dx
-
\int_\Omega\rho(\delta-\bar\delta)^2\,dx
\\
&\le
\frac12\|F_1\|_{L^\infty}
    \|\delta\|_{L^2(\rho)}^2
+
\|\nabla\omega_2\|_{L^\infty}
    \|\delta\|_{L^2(\rho)}
    \|\nabla\psi\|_{L^2(\rho)}.
\end{align*}
Using \eqref{eq:linear-H1-elliptic-difference}, we obtain for almost every
$t\in(0,T)$
\[
    \frac{d}{dt}\|\delta\|_{L^2(\rho)}^2
    \le
    C_T\|\delta\|_{L^2(\rho)}^2.
\]
Since the energy is absolutely continuous, Gronwall's inequality proves
\eqref{eq:linear-L2-stability} for every $t\in[0,T]$. If the initial data
agree, then $\delta\equiv0$, and the normalized Neumann problem gives
$\psi\equiv0$.
\end{proof}

Let now $\omega_0\in C^{0,1}(\bar\Omega)$. For every
$\alpha\in(0,1)$ we have
$\omega_0\in C^{0,\alpha}(\bar\Omega)$, so
\cref{thm:global_existence} provides a global H\"older solution.
\Cref{lem:lipschitz-propagation} upgrades this solution to the locally
spatially Lipschitz class, and \cref{thm:uniqueness} gives uniqueness together
with the weighted $L^2(\rho)$ stability estimate. Applying the H\"older
existence result with arbitrary exponents below one and using uniqueness to
identify the resulting solutions yields the regularity stated in
\cref{thm:global_wellposedness}. Hence
\cref{thm:global_wellposedness} follows.

\section{Numerical Experiments}
\label{sec:numerical-experiments}
In this section, we report the performance of weighted Laplacian flow for sampling given distributions. The details of the algorithm can be found in Appendix \ref{sec:algorithm}. Unless stated otherwise, we use $N_p=16{,}000$ particles.
For the two-dimensional examples, WLF is implemented by the grid-based
algorithm on a $128\times128$ grid. When a Kullback--Leibler (KL) curve is shown, the density of the
evolved particles is estimated by Gaussian kernel density estimation (KDE). The
KL value obtained from particles sampled directly from the target distribution
is plotted as a reference floor. The code is available at \url{https://github.com/shwangtangjun/Weighted-Laplacian-Flow}.


\subsection{Bimodal Distribution}
\label{subsec:bimodal-experiments}
The first target lives in the periodic square $[0,2\pi]^2$ and is an equal
mixture of two modes,
\begin{equation*}
    p \;\propto\; \tfrac12\,e^{4(\cos(x-\tfrac{\pi}{2})+\cos(y-\tfrac{3\pi}{2}))}
            \;+\; \tfrac12\,e^{4(\cos(x-\tfrac{3\pi}{2})+\cos(y-\tfrac{\pi}{2}))},
\end{equation*}
with well-separated centers at $(\tfrac{\pi}{2},\tfrac{3\pi}{2})$ and
$(\tfrac{3\pi}{2},\tfrac{\pi}{2})$. This non-log-concave target has a high
energy barrier between the two modes.



To study the crossing barrier capability of the proposed WLF model, 
$q_0$ is initialized as the left component. We take
$\Delta t=10^{-3}$ and $\alpha=1$. Snapshots for WLF are taken over
$t\in[0,1]$. The results are shown in Fig. \ref{fig:unimodal2bimodal}.

\begin{figure}
  \centering
  \subcaptionbox{Particle snapshots.}{\includegraphics[width=0.9\linewidth]{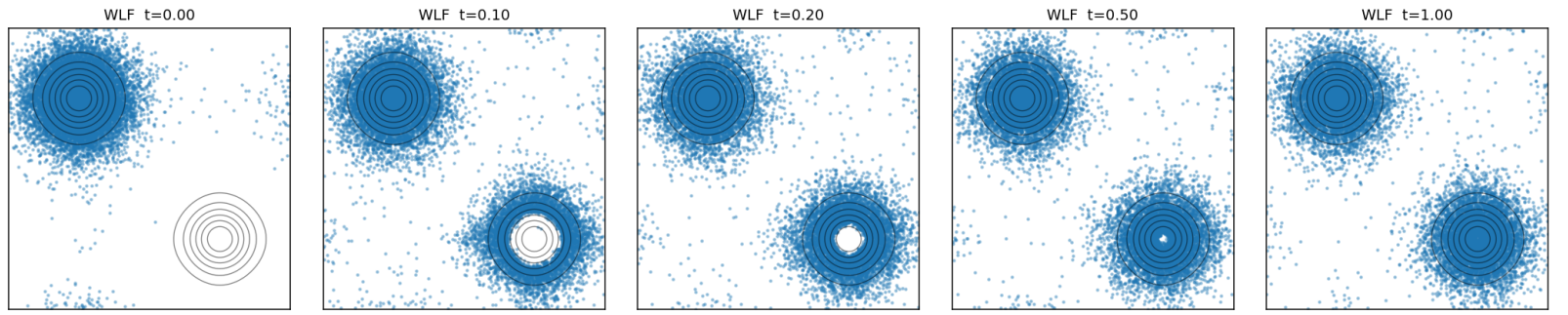}}
  \subcaptionbox{KL divergence.}{\includegraphics[width=0.4\linewidth]{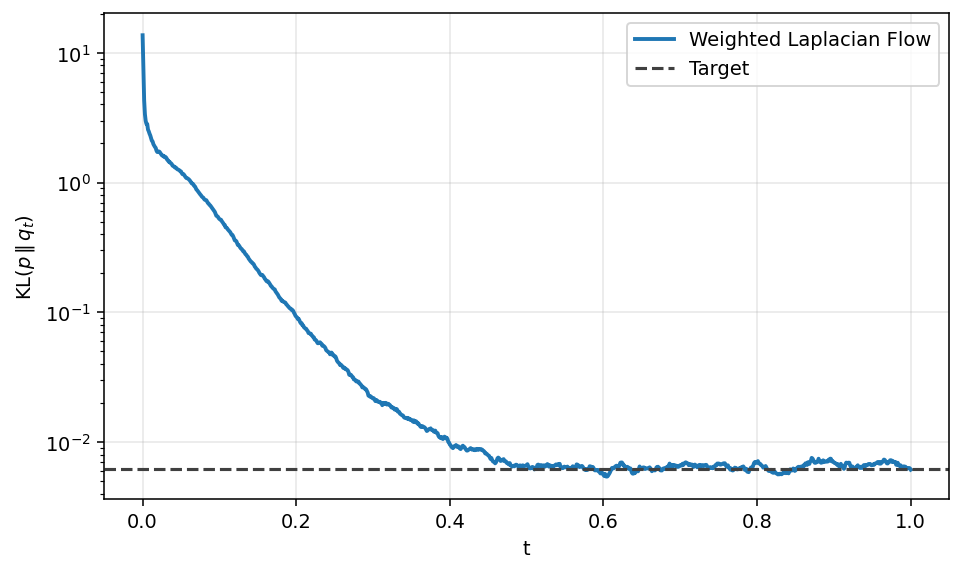}}
  \caption{Bimodal target with left-mode initialization.}
  \label{fig:unimodal2bimodal}
\end{figure}

Starting entirely in the left mode, WLF splits the cloud and transports mass
across the barrier, matching the two target modes within a short time. The
method does not rely on stochastic barrier crossing. The velocity field
$v=\nabla\phi$ creates a deterministic channel between the modes, so mass is
moved through the low-density region quickly.


\subsection{Heavy-Tailed Cauchy Distribution}
\label{subsec:cauchy-experiments}
The second target is defined on the bounded, non-periodic box $[-15,15]^2$ with
homogeneous Neumann boundary condition. It is a product of standard Cauchy
densities,
\begin{equation*}
    p(x,y)\;\propto\;\frac{1}{1+x^2}\cdot\frac{1}{1+y^2},
\end{equation*}
truncated to the box and renormalized on the grid. This target is heavy-tailed,
so $\nabla\log\rho$ is weak far from the center.

We first initialize $q_0$ uniformly on the full box and select
$\Delta t=10^{-2}$ and $\alpha=1$. Snapshots for WLF are taken over
$t\in[0,5]$. The results are shown in Fig. \ref{fig:uniform2cauchy}. WLF moves particles toward the target contours as expected.


\begin{figure}
  \centering
  \subcaptionbox{Particle snapshots.}{\includegraphics[width=0.9\linewidth]{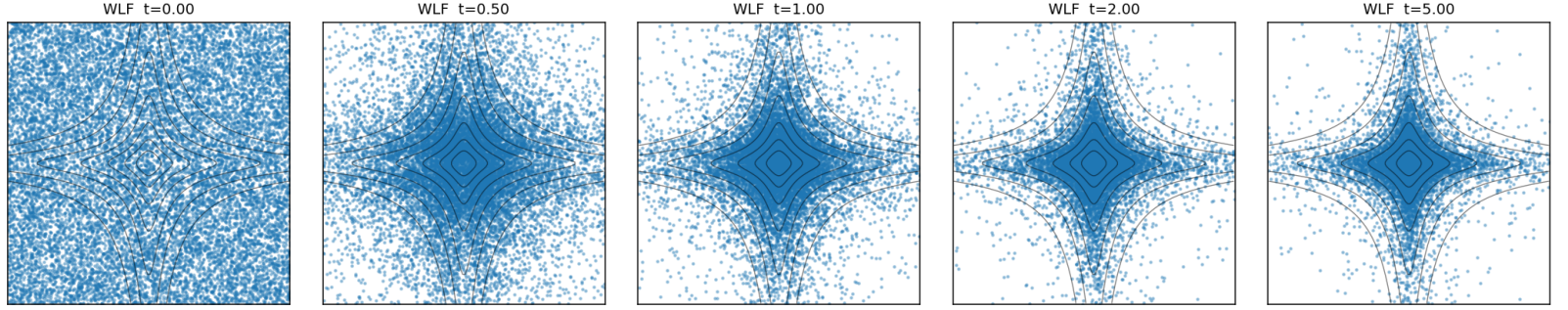}}
  \subcaptionbox{KL divergence.}{\includegraphics[width=0.4\linewidth]{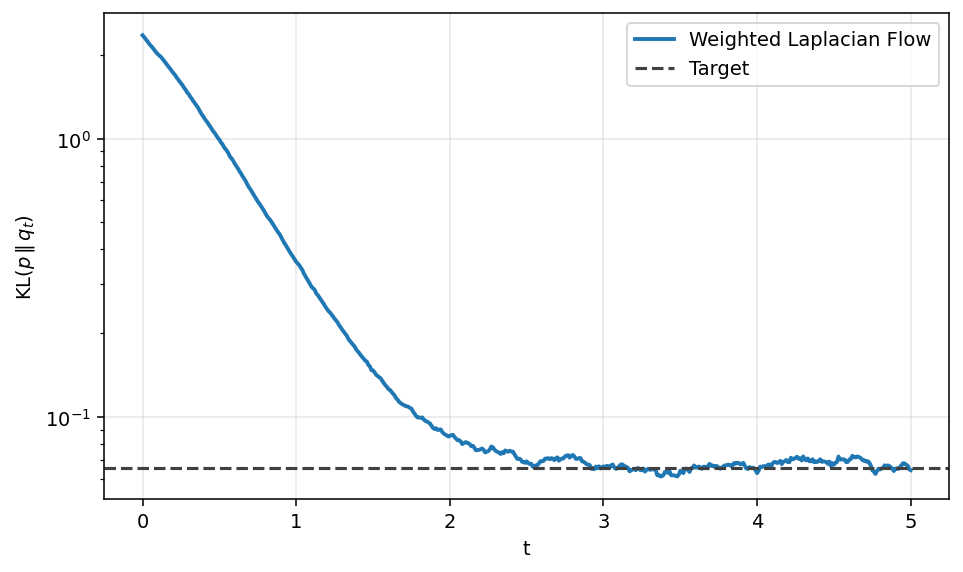}}
  \caption{Heavy-tailed target with uniform initialization.}
  \label{fig:uniform2cauchy}
\end{figure}

We also consider a more extreme situation and take $q_0$ to be uniform on a small rectangle of width $7.5$ placed in
the far corner, centered at $(10,10)$. Thus, all particles start in a low-density tail and
must be transported to the central peak and spread to the heavy shoulders. We
keep $\Delta t=10^{-2}$ and $\alpha=1$. The results are shown in Fig. \ref{fig:corner2cauchy}.

\begin{figure}
  \centering
  \subcaptionbox{Particle snapshots.}{\includegraphics[width=0.9\linewidth]{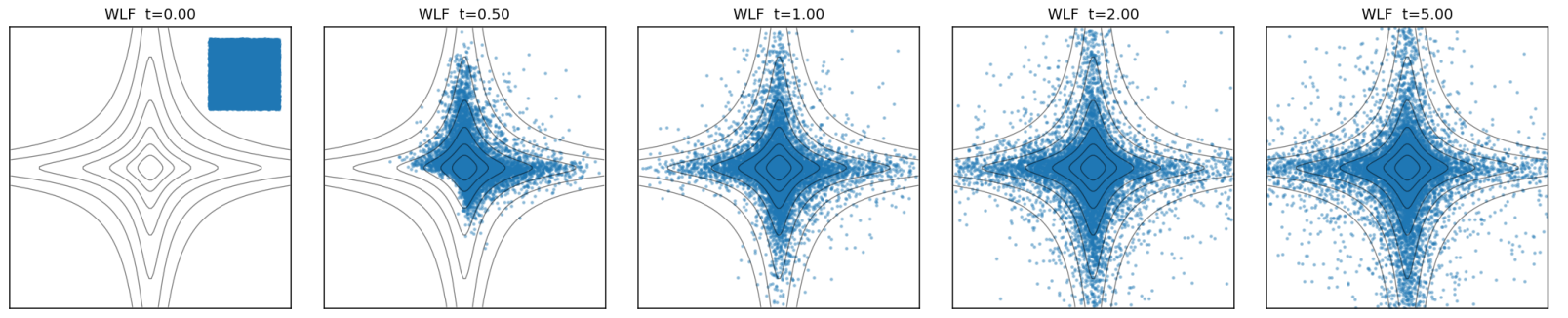}}
  \subcaptionbox{KL divergence.}{\includegraphics[width=0.4\linewidth]{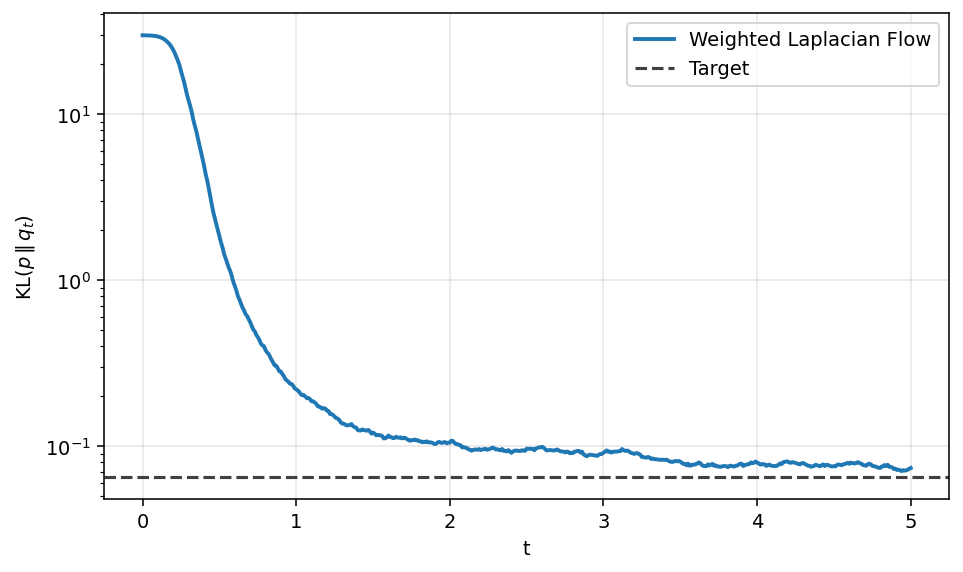}}
  \caption{Heavy-tailed target with corner initialization.}
  \label{fig:corner2cauchy}
\end{figure}

WLF is capable of handling long-range transport and heavy tails together. 
The particles are carried from the corner to the central region and then spread
along the slowly decaying Cauchy shoulders. 


\subsection{Ten-Dimensional Bimodal Distribution}
\label{subsec:10d-experiments}
We finally conduct a high dimensional test on the
ten-dimensional torus $[0,2\pi]^{10}$. We use $K=4$ Fourier modes in the
periodic network features and $M=2^{16}$ uniformly sampled collocation points in
each time step. The target is an equal mixture of two ten-dimensional product
von Mises distributions,
\begin{equation*}
 p(x)
 \propto
 \frac12\exp\left(\sum_{j=1}^{10}4\cos(x_j-\mu_{L,j})\right)
 +
 \frac12\exp\left(\sum_{j=1}^{10}4\cos(x_j-\mu_{R,j})\right),
\end{equation*}
where
\begin{equation*}
 \mu_L=(\tfrac{\pi}{2},\tfrac{3\pi}{2},\pi,\ldots,\pi),
 \qquad
 \mu_R=(\tfrac{3\pi}{2},\tfrac{\pi}{2},\pi,\ldots,\pi).
\end{equation*}
The two mixture components differ only in their first two coordinates and share
the remaining coordinates, so we visualize the first two-dimensional marginal
of the particle cloud. The KL divergence is also calculated on the first two dimensions.

The initial distribution $q_0$ is confined in the left component. We use $\Delta t=10^{-2}$ and $\alpha=1$, and show the WLF trajectory over
$t\in[0,10]$ in Fig. \ref{fig:10d-unimodel2bimodal}. WLF still works in this preliminary high-dimensional test. 



\begin{figure}
  \centering
  \subcaptionbox{Particle snapshots.}{\includegraphics[width=0.9\linewidth]{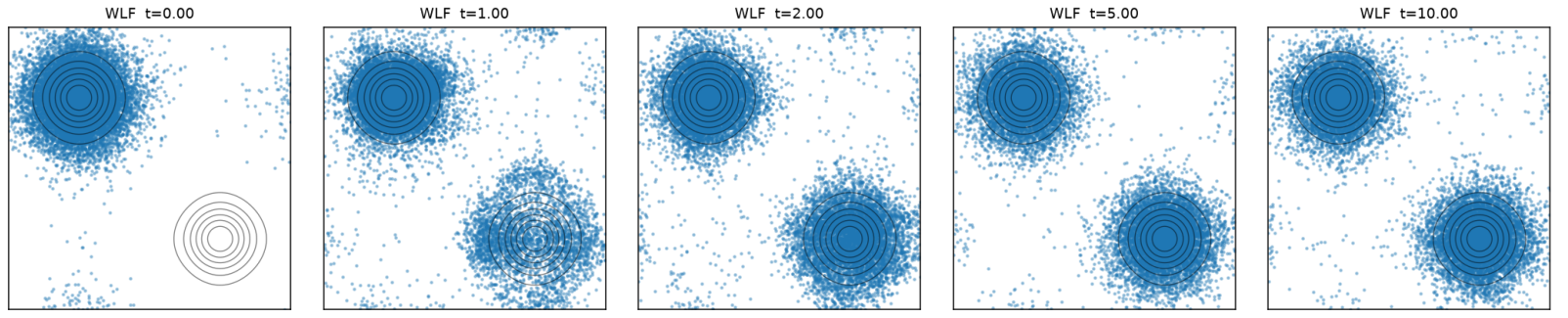}}
  \subcaptionbox{KL divergence.}{\includegraphics[width=0.4\linewidth]{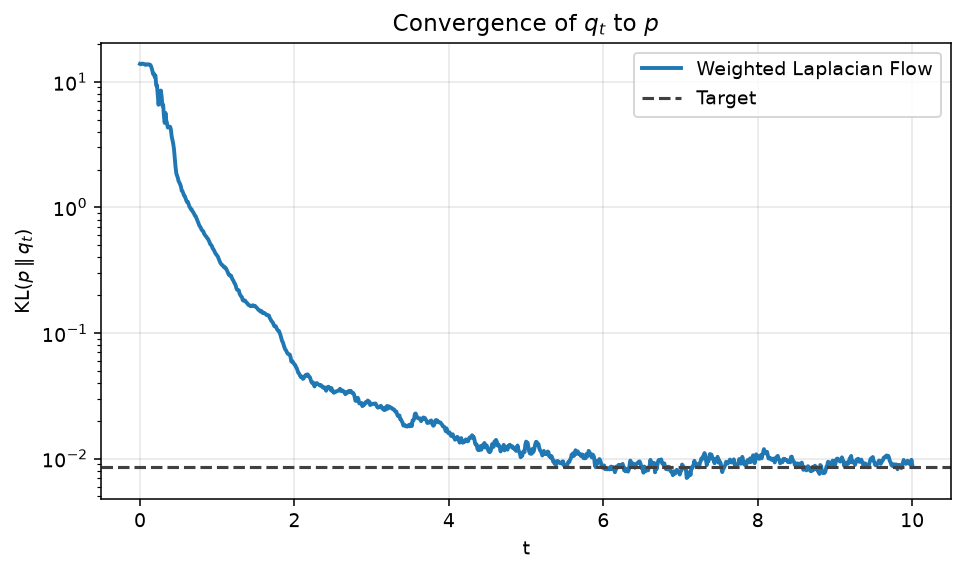}}
  \caption{Ten-dimensional bimodal target with left-mode initialization.}
  \label{fig:10d-unimodel2bimodal}
\end{figure}



\section{Conclusion}
\label{sec:conclusion}

We introduced weighted Laplacian flow (WLF), a deterministic particle-flow
framework for sampling from a target density known only up to normalization.
The dynamics are built from a centered relaxation of the logarithmic density
ratio and a target-weighted Poisson equation that converts this scalar
relaxation into a global velocity field. On bounded smooth domains with
homogeneous Neumann boundary conditions, we proved global well-posedness for
Lipschitz initial log-density ratios. The solution remains spatially Lipschitz
on every finite time interval, while a weighted $L^2(\rho)$ stability estimate
gives uniqueness and continuous dependence on the initial data.

The centered relaxation leads to an unusually explicit convergence theory.
Writing $p=\rho/\int_\Omega\rho$, the log-density-ratio oscillation satisfies
\[
    \operatorname{osc}\log\frac{p}{q_t}
    =
    e^{-t}
    \operatorname{osc}\log\frac{p}{q_0},
\]
which is equivalently exact contraction in the Hilbert projective metric.
Moreover,
\[
    \frac{d}{dt}\mathrm{KL}(q_t\,\|\,p)
    =
    -\mathrm{KL}(q_t\,\|\,p)
    -\mathrm{KL}(p\,\|\,q_t).
\]
Thus the forward relative entropy decays at least at the unit exponential
rate. These identities require neither log-concavity of the target nor a
spectral-gap estimate for the weighted Laplacian.

Beyond these dynamical identities, WLF admits a natural variational
interpretation. On the smooth manifold of positive densities, it is the formal
gradient flow of $\mathrm{KL}(p\,\|\,q)$ under a target-anchored divergence
metric. Along the flow,
\[
    \frac{d}{dt}\mathrm{KL}(p\,\|\,q_t)
    =
    -\operatorname{Var}_p\!\left(\log\frac{p}{q_t}\right).
\]
Equivalently, the variational principle first selects the steepest centered
compression of the generating entropy, and the weighted Poisson equation then
realizes that compression through the velocity field of minimum weighted
kinetic energy. Near equilibrium, the WLF metric reduces to the $L^2(p)$
geometry of relative density perturbations and the linearized equation is
simply $\partial_tu=-u$.

Several questions remain open. A natural next step is to develop the
target-anchored Hilbert structure into a global metric theory, including
geodesics, minimizing-movement schemes, and convexity properties of
$\mathrm{KL}(p\,\|\,q)$. Extending the PDE theory from bounded domains to
$\mathbb R^d$, particularly for heavy-tailed targets, is important for the
sampling applications that motivate the method. Another direction is to
establish quantitative convergence of finite-particle and numerical
approximations to the continuum flow. On the computational side, scalable
solution of the target-weighted Poisson problem remains the principal challenge
in high dimension; mesh-free elliptic solvers and discretization-error analysis
are therefore natural priorities.

\appendix

\section{Sampling Method}
\label{sec:algorithm}

We now summarize numerical realizations of WLF. Suppose
that samples $\{X_k^0\}_{k=1}^{N_p}$ from $q_0$ are available and that the
target is specified only through an unnormalized density $\rho$. We initialize
\[
    \omega^0=\log\rho-\log q_0.
\]
At every time level, $\bar\omega^n$ is treated as an additional unknown and is
determined together with the elliptic potential by the discrete Neumann
solvability condition. No normalized target density and no partition function
are used.

For an optional noise parameter $\alpha\ge0$, the particle update can be
augmented by
\[
    dX_t
    =
    \left(
        v+\alpha\nabla\log\rho-\alpha\nabla\omega
    \right)dt
    +
    \sqrt{2\alpha}\,dW_t.
\]
The deterministic WLF corresponds to $\alpha=0$.

\begin{algorithm}
\caption{Weighted Laplacian flow (grid-based)}
\label{alg:ours}
\begin{algorithmic}[1]
\REQUIRE Grid $V_h=\{s_j\}_{j=1}^M$, initial density $q_0$, particles
$\{X_k^0\}_{k=1}^{N_p}$, unnormalized target $\rho$, time step $\Delta t$,
noise parameter $\alpha\ge0$, and number of steps $N_t$
\ENSURE Particles $\{X_k^{N_t}\}_{k=1}^{N_p}$
\STATE Set
$\omega^0(s_j)=\log\rho(s_j)-\log q_0(s_j)$.
\STATE Assemble the discrete weighted Laplacian $A_\rho$ and a quadrature
weight vector $m$; cache a factorization of the augmented system when possible.
\FOR{$n=0$ to $N_t-1$}
    \STATE \textbf{Poisson step:} solve jointly for
    $(\phi^n,\bar\omega^n)$ from
    \[
    \begin{bmatrix}
        A_\rho & -\mathbf 1\\
        m^\top & 0
    \end{bmatrix}
    \begin{bmatrix}
        \phi^n\\
        \bar\omega^n
    \end{bmatrix}
    =
    \begin{bmatrix}
        -\omega^n\\
        0
    \end{bmatrix}.
    \]
    Here the last row fixes the additive constant of $\phi^n$, while
    compatibility of the first block determines $\bar\omega^n$.
    \STATE Set
    $F^n(s):=-(\omega^n(s)-\bar\omega^n)$ and
    $v^n(s):=\nabla\phi^n(s)$.
    \STATE \textbf{Transport step:} backtrace
    $\widetilde s_j^n=s_j-\Delta t\,v^n(s_j)$ and update
    \[
        \omega^{n+1}(s_j)
        =
        \omega^n(\widetilde s_j^n)
        +
        \Delta t\,F^n(\widetilde s_j^n).
    \]
    \STATE \textbf{Particle step:} interpolate $v^n$ and
    $\nabla\omega^n$ at $X_k^n$, draw
    $\xi_k^n\sim\mathcal N(0,I_d)$, and set
    \[
    \begin{aligned}
    X_k^{n+1}
    ={}&
    X_k^n+\Delta t\,v^n(X_k^n)\\
    &+\alpha\Delta t
    \left[
        \nabla\log\rho(X_k^n)-\nabla\omega^n(X_k^n)
    \right]
    +\sqrt{2\alpha\Delta t}\,\xi_k^n.
    \end{aligned}
    \]
\ENDFOR
\end{algorithmic}
\end{algorithm}

The augmented linear system is useful conceptually and computationally: the
unknown scalar $\bar\omega^n$ is obtained from solvability rather than from an
explicit normalized average. Multiplying all target weights $\rho(s_j)$ by the
same positive constant leaves the computed velocity unchanged.

\begin{algorithm}
\caption{Weighted Laplacian flow (mesh-free)}
\label{alg:ours-meshfree}
\begin{algorithmic}[1]
\REQUIRE Networks $\phi_\theta$ and $\omega_\eta$, initial density $q_0$,
particles $\{X_k^0\}_{k=1}^{N_p}$, unnormalized target $\rho$, number of
collocation points $M$, time step $\Delta t$, noise parameter $\alpha\ge0$,
and number of steps $N_t$
\ENSURE Particles $\{X_k^{N_t}\}_{k=1}^{N_p}$
\STATE Fit $\omega_{\eta^0}(z_i^{\mathrm{init}})$ to
$\log\rho(z_i^{\mathrm{init}})-\log q_0(z_i^{\mathrm{init}})$ on initial collocation points.
\FOR{$n=0$ to $N_t-1$}
    \STATE Draw collocation points $\{z_i^n\}_{i=1}^M$.
    \STATE \textbf{Poisson step:} solve jointly for the potential network
    $\phi_{\theta^n}$ and scalar $\bar\omega^n$ using the mixed weak system
    \[
    \int_\Omega
        \rho(x)\nabla\phi_{\theta^n}(x)\cdot\nabla\psi(x)\,dx
    =
    \int_\Omega
        \rho(x)(\omega_{\eta^n}(x)-\bar\omega^n)\psi(x)\,dx
    \quad\text{for all test functions }\psi,
    \]
    together with a normalization condition for $\phi_{\theta^n}$.
    In practice the integrals are replaced by Monte Carlo sums over the
    collocation points, and $\bar\omega^n$ is optimized as a scalar variable
    together with the network parameters.
    \STATE Set
    $F^n(x):=-(\omega_{\eta^n}(x)-\bar\omega^n)$.
    \STATE \textbf{Transport step:} backtrace
    $\widetilde z_i^n
    =
    z_i^n-\Delta t\,\nabla\phi_{\theta^n}(z_i^n)$,
    form
    \[
        y_i^n
        =
        \omega_{\eta^n}(\widetilde z_i^n)
        +
        \Delta t\,F^n(\widetilde z_i^n),
    \]
    and fit $\omega_{\eta^{n+1}}(z_i^n)$ to the targets $y_i^n$.
    \STATE \textbf{Particle step:} draw
    $\xi_k^n\sim\mathcal N(0,I_d)$ and set
    \[
    \begin{aligned}
    X_k^{n+1}
    ={}&
    X_k^n+\Delta t\,\nabla\phi_{\theta^n}(X_k^n)\\
    &+\alpha\Delta t
    \left[
        \nabla\log\rho(X_k^n)
        -
        \nabla\omega_{\eta^n}(X_k^n)
    \right]
    +\sqrt{2\alpha\Delta t}\,\xi_k^n.
    \end{aligned}
    \]
\ENDFOR
\end{algorithmic}
\end{algorithm}

For periodic domains, Fourier features can be used so that
$\phi_{\theta^n}$, $\omega_{\eta^n}$, and
$v^n=\nabla\phi_{\theta^n}$ are periodic by construction. The weak Poisson
formulation requires only the unnormalized weights $\rho(z_i^n)$ and treats
$\bar\omega^n$ as part of the solve, preserving the normalization-free
structure of the continuous model.

\bibliographystyle{plain}
\bibliography{reference}

\end{document}